\documentclass{birkjour}
\usepackage{mathrsfs}
\usepackage{enumerate}
 \newtheorem{thm}{Theorem}[section]
 
 \newtheorem{lem}[thm]{Lemma}
 \newtheorem{prop}[thm]{Proposition}
 \theoremstyle{definition}
 \newtheorem{defn}[thm]{Definition}
 \theoremstyle{remark}
 \newtheorem{rem}[thm]{Remark}
 
 \numberwithin{equation}{section}

\newcommand{\eps}{\varepsilon}

\newcommand{\norm}[1]{\big\|#1\big\|}
\newcommand{\abs}[1]{\left|#1\right|}
\newcommand{\sabs}[1]{|#1|}

\newcommand{\set}[1]{\left\{#1\right\}}
\newcommand{\inner}[1]{\left(#1\right)}
\newcommand{\biginner}[1]{\Bigl(#1\Bigr)}

\newcommand{\comii}[1]{\left<#1\right>}

\newcommand{\reff}[1]{(\ref{#1})}
\newcommand{\Sal}{\mathscr{S}}
\newcommand{\lpq}{L_{p,q;\lambda}}
\newcommand{\mpq}{M_{p,q;\lambda}}
\newcommand{\R}{\mathbb R}

\begin{document}

\title[Global  Gevrey hypoellipticity for the twisted Laplacian]{Global  Gevrey hypoellipticity for the twisted Laplacian on forms}

\author[W.-X. Li]{Wei-Xi Li}

\address{School of Mathematics and Statistics,  and Computational Science Hubei Key Laboratory, Wuhan University,
Wuhan 430072, China }
\email{wei-xi.li@whu.edu.cn}
\thanks{W.-X. Li  was supported by NSF of China( No. 11422106) and Fok Ying Tung Education Foundation (No. 151001).}

\author[ A.Parmeggiani]{Alberto Parmeggiani}

\address{Department of Mathematics, University of Bologna, Piazza di Porta San Donato 5, 40126 Bologna, Italy}

\email{alberto.parmeggiani@unibo.it}

\author[Y.-L. Wang]{Yan-Lin Wang}

 \address{School of Mathematics and Statistics, Wuhan University,
Wuhan 430072, China }
\email{wangyanlin@whu.edu.cn}

\subjclass{Primary 35H10; Secondary 35B65}

\keywords{Global Gevrey hypoellipticity, Twisted Laplacian, Anisotropic Gevrey regularity}
\date{}

\begin{abstract}
We study in this paper the global hypoellipticity property in the Gevrey category for the  generalized twisted Laplacian   on   forms.  
Different from the 0-form case, where the twisted Laplacian is a scalar operator, 
this is a system of differential operators when acting on forms, each component operator being elliptic locally and degenerate  globally. 
We obtain here the global hypoellipticity in anisotropic Gevrey space.  

\end{abstract}

\maketitle

\section{Introduction and main results}

The twisted Laplacian
\begin{eqnarray*}
\mathcal L =\sum_{j=1}^n \inner{D_{x_j}-
 y_j/2}^2+\sum_{j=1}^n\inner{D_{y_j}+
 x_j/2}^2
\end{eqnarray*}
(where $D=-i\partial$) is a  magnetic Schr\"odinger operator on $\mathbb R^{2n}\simeq \mathbb C^n,$ corresponding to the quantum-mechanical Hamiltonian of the motion of $n$ particles in the plane under the influence of a constant magnetic field perpendicular to the plane
itself.
The importance of the twisted Laplacian $\mathcal L$ is well-known, in that not only is  it a model of Schr\"odinger operator with a constant magnetic field, but it also describes 
the action of the reduced Heisenberg group. It is closely connected to quantum harmonic oscillators in $\mathbb R^{2n}$ and to 
the Kohn sub-Laplacian on the Heisenberg group.  Because of this, 
it has recently received a good deal of interest,  and up to now there have been extensive works on it; see for instance \cite{Thangavelu, Thangavelu1, Wong, Wong1}  and references listed therein.

The   twisted Laplacian $\mathcal L$ is well-known to be elliptic, but not globally elliptic, since its symbol is given by
\begin{eqnarray*}
	\sigma(\mathcal L)(x,y,\xi,\eta)=\sum_{j=1}^n \inner{\xi_j-
 y_j/2}^2+\sum_{j=1}^n\inner{\eta_j+
 x_j/2}^2,
 \end{eqnarray*}
 so that one cannot find a constant $C$ such that
 \begin{eqnarray*}
 	\abs{\sigma(\mathcal L)(x,y,\xi,\eta)}\geq C \inner{1+\xi^2+\eta^2+x^2+y^2}
 \end{eqnarray*}
for   $\xi^2+\eta^2+x^2+y^2$ large enough.  Hence while its local regularity properties are well understood, 
the corresponding global properties are not clear.    
There has been considerable work concerned with the spectral and global properties, for example the hypoellipticity in the 
Schwartz space or Gelfand-Shilov spaces 
\cite{DasguptaWong, GT},  the spectrum and fundamental solutions \cite{BBR,C,DasguptaWong, KochRicci, Tie},  where the key point is 
to  construct the heat kernel 
and Green's function of the twisted Laplacian $\mathcal L$ using the Weyl and the Fourier-Wigner transforms of Hermite functions.
The  global hypoellipticity in the Gevrey and analytic category for the anisotropic twisted Laplacian
\begin{eqnarray*}
 L_{p,q} =\sum_{j=1}^n \inner{D_{x_j}-
 y_j^{p_j}/2}^2+\sum_{j=1}^n\inner{D_{y_j}+
 x_j^{q_j}/2}^2=:-\sum_{j=1}^n(X_j^2+Y_j^2),
\end{eqnarray*}
where 
$$X_j=\frac{\partial}{\partial x_j}-\frac{i}{2}y_j^{p_j},\,\,\,\,Y_j=\frac{\partial}{\partial y_j}+\frac{i}{2}x_j^{q_j},$$
with $p_j, q_j\geq 1$, was obtained in \cite{WXL}.

In this paper we will investigate the global hypoellipticity  in the Gevrey class for the (anisotropic) twisted Laplacian acting on forms 
rather than functions, and in this case 
we will have a system of linear partial differential operator.  Precisely,
we consider the magnetic vector potential $A(x,y)$, represented as a $1$-form
$$\omega_A=\frac12\sum_{j=1}^n\Bigl(y_j^{p_j}dx_j-x_j^{q_j}dy_j\Bigr),$$
which gives rise to the magnetic field $B(x,y)$, represented as a $2$-form by
$$\sigma_B=d\omega_A=-\frac12\sum_{j=1}^n\Bigl(p_jy_j^{p_j-1}+q_jx_j^{q_j-1}\Bigr)dx_j\wedge dy_j=:-\sum_{j=1}^nM^{(j)}_{p,q}(x,y)dx_j\wedge dy_j.$$
Put
$$X_A:=\frac12\sum_{j=1}^n\Bigl(y_j^{p_j}\frac{\partial}{\partial x_j}-x_j^{q_j}\frac{\partial}{\partial y_j}\Bigr).$$
For $0\leq \ell\leq 2n$ a given integer, we denote from now on by $\boldsymbol{\wedge}^\ell L^2$, 
$\boldsymbol{\wedge}^\ell\mathscr{S}'$,  etc. the $\ell$ forms with coefficients in $L^2,$ $\mathscr{S}'$, etc., respectively.
Here we consider the operators on $\ell$-forms
$$D_A:=d-i\omega_A\wedge\cdot\colon\boldsymbol{\wedge}^\ell\mathscr{S}'\longrightarrow\boldsymbol{\wedge}^{\ell+1}\mathscr{S}',$$ 
with formal adjoint given by
$$D_A^*=d^*+i\,\boldsymbol{i}_{X_A}\colon\boldsymbol{\wedge}^{\ell+1}\mathscr{S}'\longrightarrow\boldsymbol{\wedge}^\ell\mathscr{S}',$$ 
where $\boldsymbol{i}_{X_A}$ denotes contraction by the vector field $X_A$. 
Consider the sequences of operators
$$\ldots\xrightarrow{D_A}\boldsymbol{\wedge}^\ell\mathscr{S}'\xrightarrow{D_A}\boldsymbol{\wedge}^{\ell+1}\mathscr{S}'\xrightarrow{D_A}\boldsymbol{\wedge}^{\ell+2}\mathscr{S}'\xrightarrow{D_A}\ldots,$$
$$\ldots\xrightarrow{D_A^*}\boldsymbol{\wedge}^{\ell+2}\mathscr{S}'\xrightarrow{D_A^*}\boldsymbol{\wedge}^{\ell+1}\mathscr{S}'\xrightarrow{D_A^*}\boldsymbol{\wedge}^\ell\mathscr{S}'\xrightarrow{D_A^*}\ldots$$
Note that they do not form a differential complex since $D_A^2\not=0$, the magnetic form being non-zero and hence a curvature term. 
We may nevertheless define  the  anisotropic twisted Laplacian of the complex as follows
\begin{equation}\label{twis}
\mathbb{L}_{p,q}=D_AD_A^*+D_A^*D_A\colon \boldsymbol{\wedge}^\ell\mathscr{S}'\longrightarrow \boldsymbol{\wedge}^\ell\mathscr{S}',	
\end{equation}
which is a system of partial differential operators.  In particular, when  $\mathbb{L}_{p,q}$ acts on functions  one can easily see that
\begin{eqnarray*}
	\mathbb{L}_{p,q}=  L_{p,q}=\sum_{j=1}^n \inner{D_{x_j}-
 y_j^{p_j}/2}^2+\sum_{j=1}^n\inner{D_{y_j}+
 x_j^{q_j}/2}^2.
\end{eqnarray*}

\begin{defn}
Let $\sigma=(\sigma_1,\cdots, \sigma_n)$ and  $\tau=(\tau_1,\cdots, \tau_n)$ with $\sigma_j, \tau_j\geq 1$ for each $j.$  We denote by   $G^{\sigma, \tau}$ the anisotropic Gevery space, 
which consists of all the smooth functions $u\in C^\infty(\mathbb R^{2n})$ for which there exists a constant $C$ such that for all 
multi-indices  $\alpha=(\alpha_1,\cdots,\alpha_n),$ $\beta=(\beta_1,\cdots\beta_n)\in\mathbb{Z}_+$ such that
\[
   \norm{\partial^{\alpha}_x\partial^\beta_y u}_{L^2} \leq C^{1+\abs\alpha+\abs\beta}(\alpha!)^\sigma(\beta!)^\tau.
\] 
By $\boldsymbol{\wedge}^\ell G^{\sigma,\tau}$ we shall denote the $\ell$-forms with values in $G^{\sigma,\tau}$.
\end{defn}

Note that   $G^{1,1,\cdots,1}$
is just the space of global analytic  functions in the whole space $\R^{2n}$. We refer to \cite{Rod} for more details on Gevrey spaces.

\begin{defn}
We say that, given an integer $1\leq \ell\leq 2n,$ a partial differential operator $P: \boldsymbol{\wedge}^\ell\mathscr{S}'(\mathbb{R}^{2n})
\longrightarrow \boldsymbol{\wedge}^\ell\mathscr{S}'(\mathbb{R}^{2n})$ is  globally $\boldsymbol{\wedge}^\ell G^{\sigma,\tau}$-hypoelliptic in $\mathbb R^{2n}$, 
if $f\in \boldsymbol{\wedge}^\ell L^2(\R^{2n})$ and $Pf\in \boldsymbol{\wedge}^\ell G^{\sigma,\tau}(\R^{2n})$  implies 
$f\in \boldsymbol{\wedge}^\ell G^{\sigma,\tau}(\R^{2n})$.
\end{defn}

Our main result can be stated as follows.

\begin{thm}\label{thG}
For all integers $\ell$ with $0\leq \ell\leq 2n$,
the twisted Laplacian $\mathbb{L}_{p,q}\boldsymbol{\wedge}^\ell\mathscr{S}'\longrightarrow\boldsymbol{\wedge}^\ell\mathscr{S}'$ 
is globally $\boldsymbol{\wedge}^\ell G^{\sigma,\tau}$-hypoelliptic in $\mathbb R^{2n}$,  where $(\sigma_j,\tau_j), j=1,2,\cdots,n$ are defined
as follows:
\begin{eqnarray*}
  \left\{
\begin{aligned}
& (\sigma_j,~\tau_j)=\big((p_j+1)/2, ~(q_j+1)/2\big),\quad    p_j=1,~{\it
   or}~q_j=1;\\[4pt]
&\sigma_j=\tau_j=\max\Big\{(p_j+1)/2, ~(q_j+1)/2\Big\},  ~~\,\,p_j,
q_j\geq 2, ~  p_j~{\it and}~q_j~{\it are~ odd};\\[4pt]
&\sigma_j=\tau_j=\max\Big\{(2p_j+2)/3, ~(2q_j+2)/3\Big\},\quad  {\it otherwise}.
\end{aligned}
\right.
\end{eqnarray*}
\end{thm}

\begin{rem}
Our result includes the global hypoellipticity property obtained in  \cite{WXL} for  $\ell=0.$  
As it will be seen below, the situation is quite different and more delicate when the twisted 
Laplacian acts on forms rather than the functions, and here we will have to deal with 
a system of degenerate differential operators. Recall $\mathbb{L}_{p,q}$  is just a scalar operator for $\ell=0.$
\end{rem}

The paper is organized as follows. In Section \ref{sec2} we derive  the representation of  twisted Laplacian in terms of a system of differential operators.   
Section \ref{sec3} is devoted to proving  the global hypoellipticity in Sobolev space and Gevrey space for a family of parametric twisted Laplacians.  
The last section is devoted to proving the main result, by exploiting the hypoellipticity property of the twisted Laplacian with parameters.

\section{The representation of the twisted laplacian as a system of differential operators} \label{sec2}

This part is devoted to deriving the representation of $\mathbb{L}_{p,q}$ in terms of a system of second order differential operators. We shall write a $k$-form $h\in\boldsymbol{\wedge}^k\mathscr{S}'$ as
$$h=\sum_{r+s=k}{\sum_{\substack{|I|=r\\|J|=s}}}'h_{IJ}dx_I\wedge dy_J,$$
where $I=(i_1,\ldots,i_r)$, $J=(j_1,\ldots,j_s)$, $|I|=i_1+i_2+\ldots+i_r$ and likewise for $|J|$ (when $|I|=0$ there is no $dx_I$ and likewise when $|J|=0$), 
the prime in the summation means that we sum over $I$ and $J$ such that
$i_1<i_2<\ldots<i_r$ and $j_1<j_2<\ldots<j_s$, and finally $dx_I=dx_{i_1}\wedge\ldots\wedge dx_{i_r}$, $dy_J=d_{j_1}\wedge\ldots\wedge dy_{j_s}$. By linearity, to compute 
$\mathbb{L}_{p,q}$ on $\boldsymbol{\wedge}^k\mathscr{S}'$ it therefore suffices to compute it on $h=h_{IJ}dx_I\wedge dy_J.$ We have the following result.

\begin{prop}
Let $h=h_{IJ}dx_I\wedge dy_J$. Then
\begin{equation*}
\begin{aligned}
\mathbb{L}_{p,q}h:&=(D_A^*D_A+D_AD_A^*)h\\\nonumber
 &=-\sum_{j=1}^n(X_j^2+Y_j^2)h_{IJ}dx_I\wedge dy_J\\\nonumber
& \quad +\sum_{j=1}^nM_{p,q}^{(j)}(x,y)h_{IJ}dy_j\wedge\boldsymbol{i}_{\partial/\partial x_j}(dx_I)\wedge dy_J\\
&  \quad-(-1)^{|I|}\sum_{j=1}^nM_{p,q}^{(j)}(x,y)h_{IJ}dx_j\wedge dx_I\wedge\boldsymbol{i}_{\partial/\partial y_j}(dy_J),
\end{aligned}
\end{equation*}
where $M^{(j)}_{p,q}(x,y)=\frac{i}{2}(q_jx_j^{q_j-1}+p_jy_j^{p_j-1})$, $1\leq j\leq n$.
Hence, on $\ell$-forms, we have
$$\mathbb{L}_{p,q}\bigl|_{\bigwedge^\ell}=L_{p,q}\otimes\mathrm{Id}_{\bigwedge^\ell}+\sum_{j=1}^nM_{p,q}^{(j)}(x,y)\otimes\Bigl(-dx_j\wedge\boldsymbol{i}_{\partial/\partial y_j}\bigl|_{\bigwedge^\ell}+dy_j\wedge
\boldsymbol{i}_{\partial/\partial x_j}\bigl|_{\bigwedge^\ell}\Bigr),$$
where, recall, $L_{p,q}=\sum_{j=1}^n\Bigl((D_{x_j}-y_j^{p_j}/2)^2 +(D_{y_j} +x_j^{q_j}/2)^2\Bigr)=-\sum_{j=1}^n(X_j^2+Y_j^2)$.
\end{prop}

\begin{proof}
We start by recalling that for $h=h_{IJ}dx_I\wedge dy_J$ we have, the operator $\boldsymbol{i}_X$ being a derivation, 
$$d^*h=-\sum_{j=1}^n\Bigl(\frac{\partial h_{IJ}}{\partial x_j}\boldsymbol{i}_{\partial/\partial x_j}(dx_I\wedge dy_J)+\frac{\partial h_{IJ}}{\partial y_j}\boldsymbol{i}_{\partial/\partial y_j}(dx_I\wedge dy_J)\Bigr)$$
$$=-\sum_{j=1}^n\Bigl(\frac{\partial h_{IJ}}{\partial x_j}\boldsymbol{i}_{\partial/\partial x_j}(dx_I)\wedge dy_J+(-1)^{|I|}\frac{\partial h_{IJ}}{\partial y_j}dx_I\wedge \boldsymbol{i}_{\partial/\partial y_j}(dy_J)\Bigr).$$
We therefore have that
\begin{eqnarray*} 
D_A^*h&=&d^*h+i\boldsymbol{i}_{X_A}(h)\\
&=&-\sum_{j=1}^n\Bigl(X_jh_{IJ}\boldsymbol{i}_{\partial/\partial x_j}(dx_I)\wedge dy_J+(-1)^{|I|}Y_jh_{IJ}dx_I\wedge\boldsymbol{i}_{\partial/\partial y_j}(dy_J)\Bigr).\end{eqnarray*}
We thus have on the one hand
\begin{eqnarray*}
&&D_A^*D_Ah\\
&=&-\sum_{j=1}^n(X_j^2+Y^2_j)h_{IJ}dx_I\wedge dy_J\\
& &+\sum_{j,k=1}^n(X_kX_jh_{IJ}dx_j+X_kY_jh_{IJ}dy_j)\wedge\boldsymbol{i}_{\partial/\partial x_k}(dx_I)\wedge dy_J\\
& &+(-1)^{|I|}\sum_{j,k=1}^n(Y_kX_jh_{IJ}dx_j+Y_kY_jh_{IJ}dy_j)\wedge dx_{I}\wedge\boldsymbol{i}_{\partial/\partial y_k}(dy_J),
\end{eqnarray*}
and, on the other,
\begin{eqnarray*}
D_AD_A^*h
&=&-\sum_{j,k=1}^nX_kX_jh_{IJ}dx_k\wedge\boldsymbol{i}_{\partial/\partial x_j}(dx_I)\wedge dy_J\\
& &+(-1)^{|I|}\sum_{j,k=1}^nX_kY_jh_{IJ}dx_k\wedge dx_I\wedge\boldsymbol{i}_{\partial/\partial y_j}(dy_J)\\
& &-\sum_{j,k=1}^nY_kX_jh_{IJ}dy_k\wedge\boldsymbol{i}_{\partial/\partial x_j}(dx_I)\wedge dy_J\\
& &+(-1)^{|I|}\sum_{j,k=1}^nY_kY_jh_{IJ}dy_k\wedge dx_I\wedge\boldsymbol{i}_{\partial/\partial y_j}(dy_J).
\end{eqnarray*}
Therefore
\begin{eqnarray*}
\mathbb{L}_{p,q}h:&=&(D_A^*D_A+D_AD_A^*)h\\
&=&-\sum_{j=1}^n(X_j^2+Y_j^2)h_{IJ}dx_I\wedge dy_J\\
& &+\sum_{j=1}^nM_{p,q}^{(j)}(x,y)h_{IJ}dy_j\wedge\boldsymbol{i}_{\partial/\partial x_j}(dx_I)\wedge dy_J\\
& &-(-1)^{|I|}\sum_{j=1}^nM_{p,q}^{(j)}(x,y)h_{IJ}dx_j\wedge dx_I\wedge\boldsymbol{i}_{\partial/\partial y_j}(dy_J),
\end{eqnarray*}
and also
$$\mathbb{L}_{p,q}\bigl|_{\bigwedge^\ell}=L_{p,q}\otimes\mathrm{Id}_{\bigwedge^\ell}+\sum_{j=1}^nM_{p,q}^{(j)}(x,y)\otimes\Bigl(-dx_j\wedge\boldsymbol{i}_{\partial/\partial y_j}\bigl|_{\bigwedge^\ell}+dy_j\wedge
\boldsymbol{i}_{\partial/\partial x_j}\bigl|_{\bigwedge^\ell}\Bigr).$$
This concludes the proof.
\end{proof}

In particular, one may write very easily the action of $\mathbb{L}_{p,q}$ on $1$-forms. In fact, if $f=\sum_{j=1}^n(f^j_1dx_j+f^j_2dy_j)\in\boldsymbol{\wedge}^1\mathscr{S}'$ from the above result we immediately get
\begin{equation}
\mathbb{L}_{p,q}f=\sum_{j=1}^n(L_{p,q}f^j_1-M_{p,q}^{(j)}(x,y)f^j_2)dx_j+\sum_{j=1}^n(L_{p,q}f^j_2+M_{p,q}^{(j)}(x,y)f^j_1)dy_j.
\label{eqL1-form}
\end{equation}

To simplify the notation we shall only consider the  case of dimension $n=1$  with  $\mathbb{L}_{p,q}$  acting on one forms, i.e., $\ell=1$. 
Due to the nature of $\mathbb{L}_{p,q}$, this is no loss of generality in our proof.

In this case, from (\ref{eqL1-form}) equation $\mathbb{L}_{p,q}f=g_1 dx+g_2 dy$ becomes the system
\begin{eqnarray} \label{lmeqnset}
  \left\{
\begin{array}{lll}
L_{p,q}f_1-M_{p,q}f_2=g_1\\
L_{p,q}f_2+M_{p,q}f_1=g_2.
\end{array}
\right.
\end{eqnarray}

\section{Twisted Laplacian with parameters}\label{sec3}
We study in this section the twisted Laplacian with  parameters, and obtain estimates in $H^\infty$ and in 
Gevrey spaces for the auxiliary parametric family attached to the original $\mathbb{L}_{p,q}.$  
Here the large parameter plays a crucial role to derive a priori estimates (see Subsection \ref{sub31} below).

Given $\lambda >0$, we set  $Z_\lambda=(Z_{1,\lambda},
Z_{2,\lambda})$ , where $Z_{j,\lambda}, j=1,2,$ are the first-order operators defined by
\begin{eqnarray*}
 Z_{1,\lambda}=D_x- \frac{\lambda^{p+1}
  }{2}y^p,\quad Z_{2,\lambda}=D_y+\frac{\lambda^{q+1}}{2}x^q.
\end{eqnarray*}
We write the twisted Laplacian with parameter as follows
\begin{equation}\label{tl}
 L_{p,q;\lambda}=\sum_{j=1}^2 Z_{j,\lambda}^2=\inner{D_x- \frac{\lambda^{p+1}
  }{2}y^p}^2+\inner{D_y+\frac{\lambda^{q+1}}{2}x^q}^2,
\end{equation}
Note that $Z_{j,\lambda}, j=1,2,$ are (formally) self-adjoint (unbounded) operators in
$L^2(\R^2)$, with common core $\mathscr{S}(\mathbb{R}^2).$
In addition, we introduce the polynomial term with parameter $\mpq$ defined by
\begin{equation}\label{mpq}
M_{p,q;\lambda}=\frac{i}{2}(q\lambda^{q+1}x^{q-1}+p\lambda^{p+1}y^{p-1}).
\end{equation}
Then the representation of $\mathbb{L}_{p,q}$ in \reff{eqL1-form} and equation \reff{lmeqnset} with parameters can be written respectively as
\begin{equation}\label{source1}
		\mathbb{L}_{p,q,\lambda} f=\lambda\inner{L_{p,q,\lambda}f_1-M_{p,q,\lambda}f_2} dx+\lambda\inner{L_{p,q,\lambda}f_2+M_{p,q,\lambda}f_1} dy,
\end{equation}
and
\begin{eqnarray} \label{eqnset}
  \left\{
\begin{array}{lll}
L_{p,q;\lambda}f_{1\lambda}-M_{p,q;\lambda}f_{2\lambda}=g_{1\lambda}\\
L_{p,q;\lambda}f_{1\lambda}+M_{p,q;\lambda}f_{2\lambda}=g_{2\lambda},
\end{array}
\right.
\end{eqnarray}
where $$f_{j\lambda}(x,y)=f_j(\lambda x,\lambda y),\quad \quad g_{j\lambda}=\lambda^2 g_j(\lambda x,\lambda y),\quad j=1,2.$$

Next, given $k\in\mathbb N,$ we consider the  space
\begin{eqnarray*}
 H^k=\set{u\in \Sal'(\mathbb{R}^2);\quad \forall\abs\alpha+\abs\beta\leq
  k,~D_x^\alpha D_y^\beta u\in L^2},
\end{eqnarray*}
and
\begin{eqnarray*}
\mathcal H_{Z_\lambda}^k=\set{u\in \Sal'(\mathbb{R}^2);\quad \forall\abs\alpha+\abs\beta\leq
  k,~ Z_{j,\lambda} D_x^\alpha D_y^\beta u\in L^2, ~ j=1,2}.
\end{eqnarray*}
Note that
\begin{eqnarray*}
H^{k+1}\bigcap\mathcal H_{Z_\lambda}^k\subset \set{u;\quad \forall\abs\alpha+\abs\beta\leq k,~   \comii x^{q} D_x^\alpha D_y^\beta u, ~ \comii
  y^{p} D_x^\alpha D_y^\beta u\in L^2}
\end{eqnarray*}
(where $\comii x=(1+|x|^2)^{1/2}$, and analogously for $\comii y$).
We shall write
\begin{eqnarray*}
 H^{\infty}=\bigcap_{k\geq 1} H^k, \quad \mathcal H_{Z_\lambda}^\infty=\bigcap_{k\geq 1}
\mathcal H_{Z_\lambda}^k,
\end{eqnarray*}
and 
\begin{eqnarray*} \,\, \norm{Z_\lambda g}_{L^2}=\inner{ \norm{Z_{1,\lambda}g}_{L^2}^2+ \norm{Z_{2,\lambda}g}_{L^2}^2}^{1/2}.
\end{eqnarray*}

Furthermore, recall that Baker-Campbell-Hausdorff formula gives (see Nier \cite[Lemma 4.14]{Nier})
\begin{gather*}
    \forall v\in \mathscr{S}(\mathbb{R}^2),~\forall I=(i_1,\cdots, i_k)\in \set{1,2}^k,\\
    \norm{ \abs{Z_{I,\lambda}}^{1/\sabs{I}} v}_{L^2}\leq C_*
 \sum_{j=1}^2\norm{Z_{j,\lambda}v}_{L^2}+C_*\norm{v}_{L^2},
\end{gather*}
where  $C_*$ is some constant independent of $\lambda$, $\sabs {I}=$length of the commutator, $Z_{I,\lambda}$ being the commutator
\[
    Z_{I,\lambda}=[Z_{i_1,\lambda},~[Z_{i_2,\lambda},\cdots, [Z_{i_{k-1},\lambda},~Z_{i_k,\lambda}]~]~].
\]
In particular, we have that
\begin{gather*}
%\begin{split}
   \forall v\in \mathscr{S}(\mathbb{R}^2),~ \forall \lambda> 0, \\
   \norm{\lambda v}_{L^2}+\norm{|\lambda^{p+1}y^{p-1}+\lambda^{q+1}x^{q-1}|^{1/2} u}_{L^2}\leq
 C_*
  \sum_{j=1}^2\norm{Z_{j,\lambda}v}_{L^2}+C_*\norm{v}_{L^2},
%  \end{split}
\end{gather*}
and that, when $p,q\geq 2$, for any given $v\in \mathscr{S}(\mathbb{R}^2)$ and any given $\lambda>0$,
\begin{equation*}
  \norm{ \lambda^{(p+1)/3}y^{(p-2)/3} v}_{L^2}+\norm{ \lambda^{(q+1)/3}x^{(q-2)/3} v}_{L^2}\leq C_*
  \sum_{j=1}^2\norm{Z_{j,\lambda}u}_{L^2}+C_*\norm{v}_{L^2}.
\end{equation*}
Consequently,  we can find two positive constants $C$  and
$\lambda_0\geq 1$, both depending only on the above $C_*$,   such that $\forall~ v\in \mathscr{S}(\mathbb{R}^2)$ and $\forall ~\lambda \geq \lambda_0,$
%\begin{gather}\label{useful1}
\begin{equation}\label{useful1}
\begin{split}
 \norm{ \inner{\lambda^{p+1}y^{p-1}+\lambda^{q+1}x^{q-1}}^{1/2}  v}_{L^2}  + \norm{\lambda v}_{L^2}\leq C
  \sum_{j=1}^2\norm{Z_{j,\lambda}v}_{L^2},
\end{split}
\end{equation}
%\end{gather}
and, when $p,q\geq2,$
\begin{equation}\label{useful2}
  \norm{ \lambda^{\frac{p+1}{3}}y^{\frac{p-2}{3}} v}_{L^2}+\norm{ \lambda^{\frac{q+1}{3}}x^{\frac{q-2}{3}} v}_{L^2}\leq C
  \sum_{j=1}^2\norm{Z_{j,\lambda}v}_{L^2}.
\end{equation}

\subsection{An a priori estimate}\label{sub31}
 Here we will prove an a priori estimate, which is crucial for the global  hypoellipiticity in Gevrey spaces as well as in the space $H^\infty.$
In what follows we set $\sum_{j=k}^n=0$ whenever $n<k.$
\begin{prop}\label{pr1}
 Let  $p,q\geq 1$ and  $\lpq$  and $\mpq$ be the operator respectively  given in
\reff{tl} and \reff{mpq}.  Then
  there exists a constant $C_0>0$,  depending only on $p, q$ and the
  constant $C$ given in \reff{useful1},  such that,     for every   integer
  $m\geq 1$   and   any given $ f_1,f_2 \in
H^\infty \bigcap \mathcal H_{Z_\lambda}^\infty$, we have:

\noindent (i) If  $p,q\geq 2$ then
\begin{eqnarray}\label{dy}
 \begin{split}
  &\sum_{i=1}^2\norm{\lambda  D_x^m f_i}_{L^2} +\sum_{i=1}^2\norm{\lambda  D_y^m f_i}_{L^2}
 +\sum_{i=1}^2\norm{Z_\lambda  D_x^m f_i}_{L^2} +\sum_{i=1}^2\norm{Z_\lambda  D_y^m f_i}_{L^2}\\
  &~\leq C_0  \inner{\norm{ D_y^m\inner{ \lpq f_1-\mpq f_2}}_{L^2}+\norm{ D_x^m\inner{ \lpq f_1-\mpq f_2}}_{L^2 }}\\
 &~+C_0  \inner{\norm{ D_y^m\inner{ \lpq f_1+\mpq f_2}}_{L^2}+\norm{ D_x^m\inner{ \lpq f_1+\mpq f_2}}_{L^2 }}\\
&~+C_0  \lambda^{p+q} \sum_{i=1}^2\inner{
m^{A_{p,q}}\norm{Z_ \lambda D_x^{m-1}   f_i}_{L^2} +
  m^{B_{p,q}}\norm{Z_ \lambda D_y^{m-1}   f_i}_{L^2}}\\
&~+C_0  \lambda^{p+q} \sum_{i=1}^2\inner{
\frac{m!}{(2q)!(m-2q)!}\norm{  \lambda D_x^{m-2q}   f_i}_{L^2} +
  \frac{m!}{(2p)!(m-2p)!}\norm{\lambda D_y^{m-2p}  f_i}_{L^2}}\\
&~+C_0  \lambda^{p+q}  \sum_{i=1}^2 \inner{\sum_{j=2}^{ 2q-1 } \frac{m!}{(m-j)!}
\norm{ Z_\lambda D_x^{m-j} f_i }_{L^2}+\sum_{j=2}^{ 2p-1 }\frac{m!}{(m-j)!}\norm{ Z_\lambda D_y^{m-j}
    f_i }_{L^2}}\\
&~+C_0  \lambda^{p+q}  \sum_{i=1}^2\inner{\sum_{j=2}^{ 2q-1 } \frac{m!}{(m-j)!}
\norm{ \lambda D_x^{m-j} f_i }_{L^2}+\sum_{j=2}^{ 2p-1 }\frac{m!}{(m-j)!}\norm{ \lambda D_y^{m-j} f_i }_{L^2}}\\
&~+C_0  \lambda^{p+q} \sum_{i=1}^2\inner{ \sum_{j=2}^{ 2q-1 } \frac{m!}{(m-j)!}
\norm{ \lambda D_x^{m-j+1} f_i }_{L^2}+\sum_{j=2}^{ 2p-1 } \frac{m!}{(m-j)!}\norm{ \lambda D_y^{m-j+1} f_i }_{L^2}}\\
&~+C_0  \lambda^{p+q}  \sum_{i=1}^2\inner{\sum_{j=3}^{ q } \frac{(m-1)!}{(m-j)!}
\norm{ Z_\lambda D_x^{m-j+1} f_i }_{L^2}+\sum_{j=3}^{ p } \frac{(m-1)!}{(m-j)!}\norm{ Z_\lambda
    D_y^{m-j+1} f_i }_{L^2}}\\
&~+C_0  \lambda^{p+q} \sum_{i=1}^2 \inner{\sum_{j=3}^{ q} \frac{(m-1)!}{(m-j)!}
\norm{ \lambda D_x^{m-j+2} f_i }_{L^2}+\sum_{j=3}^{ p } \frac{(m-1)!}{(m-j)!}\norm{ \lambda D_y^{m-j+2}
    f_i }_{L^2}},
 \end{split}
 \end{eqnarray}
%where  the
%last two summations  don't appear until $p,q\geq 3$,
the exponents $A(p,q)$ and
$B(p,q)$  in the forth line being given by
\[\left\{
\begin{array}{lll}
A(p,q)=\frac{q+1}{2}, ~B(p,q)=\frac{p+1}{2}, &\quad \textit{when both $p$ and $q$ are~ odd},\\[4pt]
A(p,q)=\frac{2q+2}{3}, ~B(p,q)=\frac{2p+2}{3}, &\quad \textit{otherwise} .
\end{array}
\right.
\]
(ii)  If  $p\geq 1$ and $q=1$, then
\begin{eqnarray}\label{dx1}
 \begin{split}
  & \sum_{i=1}^2\norm{\lambda  D_x^m f_i}_{L^2} +\sum_{i=1}^2  \norm{Z_\lambda  D_x^m f_i}_{L^2}\\
&~ \leq C_0   \norm{ D_x^m \inner{\lpq f_1-\mpq f_2 } }_{L^2}\\
 &~+C_0   \norm{ D_x^m \inner{\lpq f_2+\mpq f_1 } }_{L^2}\\
&~ +C_0\lambda^{q}  m  \sum_{i=1}^2 \norm{\lambda D_x^{m-1}   f_i}_{L^2}+C_0 \lambda^{q}
m^2 \sum_{i=1}^2 \norm{ \lambda D_x^{m-2}   f_i}_{L^2},
 \end{split}
 \end{eqnarray}
and
\begin{eqnarray}\label{dx2}
 \begin{split}
& \sum_{i=1}^2\norm{\lambda  D_y^m f_i}_{L^2} +\sum_{i=1}^2  \norm{Z_\lambda  D_y^m f_i}_{L^2}\\
 &~ \leq C_0   \norm{ D_x^m \inner{\lpq f_1-\mpq f_2 } }_{L^2}\\
 &~+C_0  \norm{ D_x^m \inner{\lpq f_2+\mpq f_1 } }_{L^2}\\
 &~+C_0 \sum_{i=1}^2 \norm{\lambda  D_x^m f_i}_{L^2}\\
 &~+C_0  \lambda^{p} \sum_{i=1}^2\inner{ m^{\inner{p+1}/2}\norm{Z_ \lambda D_y^{m-1}   f_i}_{L^2}
+\frac{m!}{(2p)!(m-2p)!}\norm{\lambda D_y^{m-2p}   f_i}_{L^2}}\\
&~+C_0  \lambda^{p} \sum_{i=1}^2  \sum_{j=2}^{ 2p-1 } \frac{m!}{(m-j)!}
\inner{ \norm{ Z_\lambda D_y^{m-j}f_i }_{L^2}+\norm{ \lambda D_y^{m-j} f_i }_{L^2}}\\
&~++C_0  \lambda^{p} \sum_{i=1}^2  \sum_{j=2}^{ 2p-1 } \frac{m!}{(m-j)!}
\inner{\norm{ \lambda D_y^{m-j+1} f_i }_{L^2}+\norm{ \lambda D_x^{m-j+1} f_i }_{L^2}}.
%&~+C_0  \lambda^{p} \sum_{i=1}^2  \sum_{j=3}^{ p } \frac{(m-1)!}{(m-j)!}
%\inner{\norm{ Z_\lambda D_y^{m-j+1} v }_{L^2}+\norm{ \lambda D_y^{m-j+2} f_i }_{L^2}}.
\end{split}
\end{eqnarray}
 
\end{prop}

This proposition can be deduced from  the following series of lemmas.
\begin{lem}\label{lemma1}
Let  $p,q\geq 1$ and  let $\lpq$ and $\mpq$ be the operator  respectively given in
\reff{tl}and \reff{mpq}.  Then
  there exists a constant $ C_0$,  depending only on $p, q$ and the
  constant $C$ given in \reff{useful1},  such that  for every integer $m\geq
  1$   and   any  given $f_1,f_2\in
H^\infty \bigcap \mathcal H_{Z_\lambda}^\infty$, we have
\begin{eqnarray}\label{dy+}
 \begin{split}
  &\norm{\lambda  D_y^m f_1}_{L^2} +\norm{Z_\lambda  D_y^m f_1}_{L^2}\\
 &~\leq  C_0  \norm{ D_y^m\inner{ \lpq f_1-\mpq f_2}}_{L^2}\\
 &~+C_0m\norm{ \lambda^{p+1} y^{p-1}D_y^{m-1}   f_1}_{L^2}\\
&~+C_0m\norm{ \lambda^{p+1} y^{p-1}D_y^{m-1}   f_2}_{L^2}+C_0m\lambda^p \norm{\lambda D_y^{m-1}f_2}_{L^2}\\
&~+C_0  \lambda^{p}\frac{m!} {(2p)!(m-2p)!}\norm{\lambda  D_y^{m-2p}   f_1}_{L^2}\\
&~+C_0   \sum_{j=2}^{2p-1}\frac{m!}{j!(m-j)!}   \norm{\lambda^{p+1} y^{\delta_j}D_y^{m-j}  f_1 }_{L^2}\\
&~+ C_0   \sum_{j=3}^{p}\frac{(m-1)!}{j!(m-j)!}   \norm{\lambda^{p+1}y^{\rho_j }D_y^{m-j+1}  f_1 }_{L^2}\\
&~+C_0   \sum_{j=2}^{p-1}\frac{(m-1)!}{j!(m-j)!}   \norm{\lambda^{p+1}y^{\eta_j }D_y^{m-j}  f_2 }_{L^2}\\
&~+C_0\frac{1}{\lambda}\norm{Z_\lambda  D_y^m f_2}_{L^2},
 \end{split}
 \end{eqnarray}
where  $\delta_j, \rho_j,\eta_j\in\set{1,2,\cdots, p-1}$.
Similarly we have
\begin{eqnarray}\label{dx+}
 \begin{split}
 & \norm{\lambda  D_x^m f_1}_{L^2} +\norm{Z_\lambda  D_x^m f_1}_{L^2}\\
 &~\leq  C_0  \norm{ D_x^m\inner{ \lpq f_1-\mpq f_2}}_{L^2}\\
 &~+C_0m\norm{ \lambda^{q+1} x^{q-1}D_x^{m-1}   f_1}_{L^2}\\
&~+C_0m\norm{ \lambda^{q+1} x^{q-1}D_x^{m-1}   f_2}_{L^2}+C_0m\lambda^q \norm{\lambda D_x^{m-1}f_2}_{L^2}\\
&~+C_0  \lambda^{q}\frac{m!} {(2p)!(m-2p)!}\norm{\lambda  D_x^{m-2q}   f_1}_{L^2}\\
&~+C_0   \sum_{j=2}^{2q-1}\frac{m!}{j!(m-j)!}   \norm{\lambda^{q+1} x^{\delta_j}D_x^{m-j}  f_1 }_{L^2}\\
&~+ C_0   \sum_{j=3}^{q}\frac{(m-1)!}{j!(m-j)!}   \norm{\lambda^{q+1}x^{\rho_j }D_x^{m-j+1}  f_1 }_{L^2}\\
&~+C_0   \sum_{j=2}^{q-1}\frac{(m-1)!}{j!(m-j)!}   \norm{\lambda^{q+1}x^{\eta_j }D_x^{m-j}  f_2 }_{L^2}\\
&~+C_0\frac{1}{\lambda}\norm{Z_\lambda  D_x^m f_2}_{L^2},
\end{split}
 \end{eqnarray}
where  $\delta_j, \rho_j,\eta_j\in\set{1,2,\cdots, p-1}$ .
\end{lem}

\begin{proof}
Here we only prove \reff{dy+} since \reff{dx+} can be handled in a similar way. In the sequel $C_p$ will denote different
suitable constants depending  only on $p$ and the constant $C$ in \reff{useful1}. 
Leibniz formula gives
\begin{eqnarray*}
&&\lpq D_y^{m}  f_1-\mpq D_y^{m}  f_2\\
&=&D_y^{m} [\lpq f_1-\mpq f_2] \\&&-\sum_{j=1}^{  2p }{m\choose j }\inner{ D_y^j\inner{\lambda^{2(p+1)}y^{2p}/4} -
  D_y^j\inner{\lambda^{p+1}y^{p}} D_x  }D_y^{m-j} f_1\\&&+\frac{i}{2}\sum_{j=1}^{p-1}{m\choose j}\inner{D_y^j\inner{p\lambda^{p+1}y^{p-1}}}D_y^{m-j}f_2.
\end{eqnarray*}
Furthermore, a direct computation of the terms in the parentheses of the above equality shows that
\begin{eqnarray*}
 && D_y^j\inner{\lambda^{2(p+1)}y^{2p}/4} - D_y^j\inner{\lambda^{p+1}y^{p}} D_x \\
& =&\left\{
\begin{array}{lll}
i\lambda^{p+1} py^{p-1}\inner{D_x-\lambda^{p+1} y^p/2},\quad &j=1,\\
a_{p,j}\lambda^{p+1}  y^{p-j}\inner{D_x-\lambda^{p+1}  y^p/2}+b_{p,j} \lambda^{2(p+1)}  y^{2p-j},\quad &2\leq j\leq
p,\\
c_{p,j}\lambda^{2(p+1)}  y^{2p-j},\quad &p< j\leq 2p,
\end{array}
\right.
\end{eqnarray*}
and
\begin{eqnarray*}
D_y^j\inner{p\lambda^{p+1}y^{p-1}}
=\left\{
\begin{array}{lll}
p(p-1)\lambda^{p+1}y^{p-2},\quad &j=1,\\
d_{p,j}\lambda^{p+1}y^{p-1-j},\quad &2\leq j\leq p-1,
\end{array}
\right.
\end{eqnarray*}
where  $a_{p,j}, b_{p,j}, c_{p,j},d_{p,j}$ are   constants depending only on $p$ and
$j$ such that
\begin{equation}\label{abcd}
  2( \sabs{ a_{p,j}}+\sabs{b_{p,j}}+\sabs{c_{p,j}}+\sabs{d_{p,j}})\leq C_p.
\end{equation}
Then
\begin{eqnarray*}
&&\lpq D_y^{m}  f_1-\mpq D_y^{m}  f_2\\
&=&D_y^{m} [\lpq f_1-\mpq f_2]\\
&&-im p\lambda^{p+1} y^{p-1}\inner{D_x-\lambda^{p+1} y^p/2}D_y^{m-1}f_1 \\
&&-\sum_{j=2}^{ p}{m\choose
  j }a_{p,j} \lambda^{p+1} y^{p-j}\inner{D_x-\lambda^{p+1} y^p/2} D_y^{m-j} f_1\\
&&-\sum_{j=2}^{p}{m\choose
  j }\lambda^{2(p+1)} b_{p,j} y^{2p-j} D_y^{m-j} f_1 \\
&&-\sum_{j=p+1}^{  2p }{m\choose
  j }\lambda^{2(p+1)} c_{p,j} y^{2p-j}  D_y^{m-j} f_1\\
&&+\frac{i}{2}mp(p-1)\lambda^{p+1}y^{p-2}D_y^{m-1}f_2\\
&&+\frac{i}{2}\sum_{j=2}^{p-1}{m\choose j}d_{p,j}\lambda^{p+1}y^{p-1-j}D_y^{m-j}f_2.
\end{eqnarray*}
Taking the $L^2$-inner product with $D_y^mf_1$ on both sides of the above equality and by virtue of \reff{useful1},
one has
\begin{equation}\label{estimate1}
 \begin{aligned}
  &\norm{\lambda D_y^{m} f_1 }^2 +\norm{(D_x-\lambda^{p+1}y^p/2) D_y^{m} f_1 }^2+ \norm{(D_y+\lambda^{q+1}x^q/2)  D_y^{m} f_1 }^2\\
 \leq &~C_p \sum_{1\leq j\leq 8}\sabs{I_j},
 \end{aligned}
\end{equation}
where
\begin{eqnarray*}
 I_1&=&\inner{D_y^{m} [\lpq f_1 -\mpq f_2], ~D_y^{m}  f_1}_{L^2},\\
I_2&=&-i\,m\,p\, \biginner{\lambda^{p+1} y^{p-1}\inner{D_x-\lambda^{p+1} y^p/2} D_y^{m-1} f_1,  ~D_y^{m}  f_1
}_{L^2},\\
I_3&=&-\sum_{j=2}^{ p}{m\choose
  j }a_{p,j} \biginner{\lambda^{p+1}y^{p-j}\inner{D_x-\lambda^{p+1}y^p/2} D_y^{m-j} f_1,  ~D_y^{m}  f_1
}_{L^2},\\
I_4&=&-\sum_{j=2}^{p}{m\choose
  j }b_{p,j} \biginner{\lambda^{2(p+1)} y^{2p-j}  D_y^{m-j} f_1,  ~D_y^{m} f_1
}_{L^2},
\\
I_5&=&-\sum_{j=p+1}^{ 2p }{m\choose
  j }c_{p,j} \biginner{\lambda^{2(p+1)} y^{2p-j}  D_y^{m-j} f_1,  ~D_y^{m}  f_1}_{L^2}\\
I_6&=&\frac{i}{2}mp(p-1)\biginner{\lambda^{p+1}y^{p-2}D_y^{m-1}f_2,~D_y^{m}  f_1}_{L^2}\\
I_7&=&\frac{i}{2}\sum_{j=2}^{p-1}{m\choose j}d_{p,j}\biginner{\lambda^{p+1}y^{p-1-j}D_y^{m-j}f_2,~D_y^{m}  f_1}_{L^2}\\
I_8&=&\biginner{\mpq D_y^m f_2,~D_y^{m}  f_1}_{L^2}.
\end{eqnarray*}
%%%%%%%%%%%%%%%%%%%%%%%%%%I_1---I_5
By \reff{abcd}, using the Cauchy-Schwarz inequality one has that for any given
$\eps>0$ there exists a constant $C_{\eps,p}$, depending only on $\eps$ and
$p$, such that
\begin{eqnarray*}
|I_1| &\leq & \norm{D_y^{m} [\lpq f_1 -\mpq f_2]}_{L^2} \norm{D_y^{m} f_1 }_{L^2}\\
&\leq &\eps \norm{ \lambda D_y^{m} f_1 }_{L^2}^2 + \lambda^{-1}C_{\eps,p}\norm{D_y^{m} [\lpq f_1 -\mpq f_2]}_{L^2}^2,
\end{eqnarray*}
and
\begin{eqnarray*}
|I_2|+|I_3| & \leq & ~\eps~  \norm{(D_x-\lambda^{p+1}y^p/2) D_y^{m} f_1
}^2+C_{\eps,p} m^2  \norm{ \lambda^{p+1} y^{p-1}D_y^{m-1}
  f_1}_{L^2}^2\\
&&+C_{\eps,p}\sum_{j=2}^p {m\choose
  j }^2
%\inner{\frac{m!}{j!(m-j)!}}^2
\norm{ \lambda^{p+1} y^{p-j}D_y^{m-j}  f_1}_{L^2}^2.
\end{eqnarray*}
As for the  term $I_4$ we have
\begin{eqnarray*}%\label{45}
|I_4| &\leq& ~C_p m^2\norm{ \lambda^{p+1}  y^{p-1} D_y^{m-1}  f_1
}_{L^2}^2\\
&&+C_p \sum_{j=2}^{p}\inner{\frac{(m-1)!}{j!(m-j)!}}^2  \norm{\lambda^{p+1} y^{p-j} D_y^{m-j} f_1}_{L^2}^2\\
&&+C_p \sum_{j=3}^{p}\inner{\frac{(m-1)!}{j!(m-j)!}}^2 \norm{\lambda^{p+1} y^{p-j+1} D_y^{m-j+1} f_1}_{L^2}^2,
\end{eqnarray*}
which can be deduced from the following integration by parts
\begin{eqnarray*}
I_4&=&-\sum_{j=2}^{p}{m\choose
  j }b_{p,j} \biginner{\inner{D_y(\lambda^{2(p+1)}y^{2p-j})}  D_y^{m-j} f_1,  ~D_y^{m-1}  f_1}_{L^2}\\
&&-\sum_{j=2}^{p}{m\choose j }b_{p,j} \biginner{\lambda^{2(p+1)}y^{2p-j}  D_y^{m-j+1} f_1,  ~D_y^{m-1}  f_1}_{L^2}\\
&=&-\sum_{j=2}^{p}{m\choose j }b_{p,j} (2p-j)\biginner{\lambda^{2(p+1)} y^{2p-j-1}  D_y^{m-j} f_1,  ~D_y^{m-1} f_1}_{L^2}\\
&&-b_{p,2}\frac{m(m-1)}{2}\norm{\lambda^{p+1}y^{p-1}D_y^{m-1}f_1}_{L^2}^2\\
&&-\sum_{j=3}^{p}{m\choose j }b_{p,j} \biginner{\lambda^{2(p+1)}y^{2p-j}  D_y^{m-j+1} f_1,  ~D_y^{m-1}  f_1}_{L^2}.
\end{eqnarray*}
Similarly we get for $I_5$ that
\begin{eqnarray*}%\label{45}
%\begin{split}
|I_5| &\leq& ~C_p m^2\norm{ \lambda^{p+1}   D_y^{m-1}  f_1
}_{L^2}^2\\
&&+C_p \sum_{j=p+1}^{2p}\inner{\frac{(m-1)!}{j!(m-j)!}}^2  \norm{\lambda^{p+1} y^{2p-j} D_y^{m-j} f_1}_{L^2}^2\\
&&+C_p \sum_{j=p+1}^{2p}\inner{\frac{(m-1)!}{j!(m-j)!}}^2 \norm{\lambda^{p+1} y^{p-j+1} D_y^{m-j+1} f_1}_{L^2}^2.
%\end{split}
\end{eqnarray*}
%%%%%%%%%%%%%%%%%%%%%%%%%%%%%%%
As regards the term $I_6$, when $p=1$ it vanishes. We may therefore consider $I_6$ when $p\geq 2$ and have
\begin{eqnarray*}
|I_6| &\leq & mp(p-1)\norm{D_y^{m}  f_1 }_{L^2} \norm{\lambda^{p+1}y^{p-2}D_y^{m-1}f_2}_{L^2}\\
&\leq &\eps \norm{ \lambda D_y^{m} f_1 }_{L^2}^2 + m^2C_{\eps,p}\norm{\lambda^{p+1}y^{p-2}D_y^{m-1}f_2}_{L^2}^2\\
&\leq &\eps \norm{ \lambda D_y^{m} f_1 }_{L^2}^2 + m^2C_{\eps,p}\lambda^{2p}\norm{\lambda D_y^{m-1}f_2}_{L^2}^2\\
&&+m^2C_{\eps,p}\norm{\lambda^{p+1} y^{p-1} D_y^{m-1}f_2}_{L^2}^2,
\end{eqnarray*}
the last inequality holding because
 \begin{eqnarray*}
&&m^2\norm{\lambda^{p+1}y^{p-2}D_y^{m-1}f_2}_{L^2}^2\\
&=&m^2\inner{\int_{\sabs y \leq 1} +\int_{\sabs y > 1} }\lambda^{2(p+1)}y^{2(p-2)}\sabs{D_y^{m-1}f_2}^2 dxdy\\
&\leq &m^2\int_{\sabs y >1}\lambda^{2(p+1)}y^{2(p-1)}\sabs{D_y^{m-1}f_2}^2 dxdy \\&&+m^2\int_{\sabs y\leq 1}\lambda^{2(p+1)}\sabs{D_y^{m-1}f_2}^2 dxdy \\
&\leq&m^2\lambda^{2p}\norm{\lambda D_y^{m-1}f_2}_{L^2}^2+m^2\norm{\lambda^{p+1} y^{p-1} D_y^{m-1}f_2}_{L^2}^2.
\end{eqnarray*}
By arguing as for the term $I_1$ we estimate the term $I_7$ as
\begin{eqnarray*}
|I_7|& \leq & C_p \sum_{j=2}^{p-1}{m\choose j }\norm{\lambda^{p+1}y^{p-1-j}D_y^{m-j}f_2}_{L^2}\norm{ \lambda D_y^{m} f_1 }_{L^2}\\
&\leq & \eps \norm{ \lambda D_y^{m} f_1 }_{L^2}^2 + C_{\eps,p}\sum_{j=2}^{p-1}{m\choose j }^2\norm{\lambda^{p+1}y^{p-1-j}D_y^{m-j}f_2}_{L^2}^2.
\end{eqnarray*}
Finally, for the term $I_8$ we have
\begin{eqnarray*}
|I_8|&= &\Bigl|\frac{i}{2}\int \inner{q\lambda^{q+1}x^{q-1}+p\lambda^{p+1}x^{p-1}}D_y^{m}f_2 D_y^{m} f_1 dxdy\Bigr|\\
&\leq&C_{p,q}\int \sabs{\lambda^{q+1}x^{q-1}+\lambda^{p+1}x^{p-1}}\sabs{D_y^{m}f_2} \sabs{D_y^{m} f_1} dxdy\\
&\leq&C_{p,q}\norm{\lambda \sabs{\lambda^{q+1}x^{q-1}+\lambda^{p+1}x^{p-1}}^{\frac{1}{2}}D_y^{m} f_1}_{L^2}\\
&&\times\norm{\frac{1}{\lambda}\sabs{\lambda^{q+1}x^{q-1}+\lambda^{p+1}x^{p-1}}^{\frac{1}{2}}D_y^{m} f_2}_{L^2}\\
&\leq& \eps \lambda^2\norm{\sabs{\lambda^{q+1}x^{q-1}+\lambda^{p+1}x^{p-1}}^{\frac{1}{2}}D_y^{m} f_1}_{L^2}^2\\
&&+C_{p,q,\eps}\frac{1}{\lambda^2}\norm{\sabs{\lambda^{q+1}x^{q-1}+\lambda^{p+1}x^{p-1}}^{\frac{1}{2}}D_y^{m} f_2}_{L^2}^2\\
&\leq&\eps\lambda^2C\norm{Z_{\lambda}D_y^m f_1}_{L^2}^2+C_{p,q,\eps}\frac{1}{\lambda^2}C\norm{Z_{\lambda}D_y^m f_2}_{L^2}^2,
\end{eqnarray*}
where $C_{p,q}$ and $C_{p,q,\eps}$ are constants respectively depending only on $p,q$ and $p,q,\eps$. The last inequality follows from \reff{useful1}.

Combining the estimates of $I_1,I_2,\cdots, I_8$ and \reff{estimate1},
the  result  \reff{dy+} follows immediately if we let $\eps$ above be
sufficiently small (such that $\eps \leq\frac{1}{\lambda^4}$) and $\lambda$ large enough. The proof is thus complete.
\end{proof}

\begin{lem}\label{lemma2}
Let  $p,q\geq 1$ and  let $\lpq$ and $\mpq$ be the operator  respectively given in
\reff{tl} and \reff{mpq}.  There exists a constant $ C_0$,  depending only on $p, q$ and the
  constant $C$ given in \reff{useful1},  such that  for every integer $m\geq
  1$   and   any  given $f_1,f_2\in
H^\infty \bigcap \mathcal H_{Z_\lambda}^\infty$, we have
\begin{eqnarray}\label{dy2+}
 \begin{split}
  &\norm{\lambda  D_y^m f_2}_{L^2}+\norm{Z_\lambda  D_y^m f_2}_{L^2} \\
 &~\leq  C_0  \norm{ D_y^m\inner{ \lpq f_2+\mpq f_1}}_{L^2}\\
 &~+C_0m\norm{ \lambda^{p+1} y^{p-1}D_y^{m-1}   f_2}_{L^2}\\
&~+C_0m\norm{ \lambda^{p+1} y^{p-1}D_y^{m-1}   f_1}_{L^2}+C_0m\lambda^p \norm{\lambda D_y^{m-1}f_1}_{L^2}\\
&~+C_0  \lambda^{p}\frac{m!} {(2p)!(m-2p)!}\norm{\lambda  D_y^{m-2p}   f_2}_{L^2}\\
&~+C_0   \sum_{j=2}^{2p-1}\frac{m!}{j!(m-j)!}   \norm{\lambda^{p+1} y^{\delta_j}D_y^{m-j}  f_2 }_{L^2}\\
&~+ C_0   \sum_{j=3}^{p}\frac{(m-1)!}{j!(m-j)!}   \norm{\lambda^{p+1}y^{\rho_j }D_y^{m-j+1}  f_2}_{L^2}\\
&~+C_0   \sum_{j=2}^{p-1}\frac{(m-1)!}{j!(m-j)!}   \norm{\lambda^{p+1}y^{\eta_j }D_y^{m-j}  f_1}_{L^2}\\
&~+C_0\frac{1}{\lambda}\norm{Z_\lambda  D_y^m f_1}_{L^2},
 \end{split}
 \end{eqnarray}
where  $\delta_j, \rho_j,\eta_j\in\set{1,2,\cdots, p-1}$ .
%, and the last two summations above don't appear  until   $p\geq 2$ and the
%last one doesn't appear until $p\geq 3$.
Similarly we have
\begin{eqnarray}\label{dx2+}
 \begin{split}
  &\norm{\lambda  D_x^m f_2}_{L^2} +\norm{Z_\lambda  D_x^m f_2}_{L^2} \\
 &~\leq  C_0  \norm{ D_x^m\inner{ \lpq f_2+\mpq f_1}}_{L^2}\\
 &~+C_0m\norm{ \lambda^{q+1} x^{q-1}D_x^{m-1}   f_2}_{L^2}\\
&~+C_0m\norm{ \lambda^{q+1} x^{q-1}D_x^{m-1}   f_1}_{L^2}+C_0m\lambda^q \norm{\lambda D_x^{m-1}f_1}_{L^2}\\
&~+C_0  \lambda^{q}\frac{m!} {(2p)!(m-2p)!}\norm{\lambda  D_x^{m-2q}   f_2}_{L^2}\\
&~+C_0   \sum_{j=2}^{2q-1}\frac{m!}{j!(m-j)!}   \norm{\lambda^{q+1} x^{\delta_j}D_x^{m-j}  f_2 }_{L^2}\\
&~+ C_0   \sum_{j=3}^{q}\frac{(m-1)!}{j!(m-j)!}   \norm{\lambda^{q+1}x^{\rho_j }D_x^{m-j+1}  f_2 }_{L^2}\\
&~+C_0   \sum_{j=2}^{q-1}\frac{(m-1)!}{j!(m-j)!}   \norm{\lambda^{q+1}x^{\eta_j }D_x^{m-j}  f_1 }_{L^2}\\
&~+C_0\frac{1}{\lambda}\norm{Z_\lambda  D_x^m f_1}_{L^2},
\end{split}
 \end{eqnarray}
where  $\delta_j, \rho_j,\eta_j\in\set{1,2,\cdots, p-1}$.
%, and the last two summations above don't appear  until   $q\geq 2$ and the
%last one doesn't appear until $q\geq 3$.
\end{lem}
\begin{proof}
The proof is similar to that of Lemma \ref{lemma1} and will be omitted.
\end{proof}

%%%%%%%%%%%%%%%%%%%%%%%%%%%%%%%some lemmas have been proved by Wei-Xi Li and Alberto Parmeggiani

We next recall the following three lemmas, which have been proven by W.-X.~Li and A.~Parmeggiani in  \cite{WXL}.
\begin{lem}\label{lamma3}
Let $q=1, p\geq 1.$   For any given $\eps>0,$  there exists a constant $C_{\eps}$, which depends only on
$\eps$,  $p$ and the constant $C$ in \reff{useful1},  such that for
every $m\geq 1$ and any $v\in H^\infty$ we have
\begin{eqnarray}\label{inter}
\begin{split}
& m \norm{\lambda^{p+1}  y^{p-1}D_y^{m-1}  v}_{L^2}  \\
\leq&~\eps  \norm{ \lambda D_x^{m}  v} +\eps   \norm{\lambda D_y^{m}  v}+C_{\eps} m \lambda^{p}
\norm{\lambda D_y^{m-1}  v} \\
& +C_{\eps} \lambda m^{(p+1)/2}\norm{Z_\lambda D_y^{m-1}  v}.
\end{split}
\end{eqnarray}
\end{lem}

\begin{lem}\label{lemma4}
Let $p,q\geq 2.$  Then for any given $\eps>0,$  there exists a constant $C_{\eps}$, which depends only
on $\eps, p$ and the constant $C$ in \reff{useful2}, such that for
every $m\geq 1$ and all $v\in H^\infty$ we have
\begin{equation}\label{inter1}
\begin{aligned}
 m \norm{\lambda^{p+1}  y^{p-1}D_y^{m-1}  v}_{L^2}  \leq&   \eps   \inner{ \norm{D_y^{m}  v}_{L^2}+ \norm{ D_x^{m}  v}_{L^2}}\\
 &+C_{\eps}\lambda^{p}m^{(2p+2)/3} \norm{
   Z_\lambda D_y^{m-1}  v}_{L^2}.
   \end{aligned}
\end{equation}
In particular, if both $p$ and $q$ are odd then
\begin{equation}\label{inter2}
\begin{aligned}
 m \norm{\lambda^{p+1}  y^{p-1}D_y^{m-1}  v}_{L^2}  \leq&   \eps   \inner{ \norm{D_y^{m}  v}_{L^2}+ \norm{ D_x^{m}  v}_{L^2}}\\
 &+C_{\eps}\lambda^{p}m^{(p+1)/2} \norm{
   Z_\lambda D_y^{m-1}  v}_{L^2}.
   \end{aligned}
\end{equation}
\end{lem}

%\begin{proof}
% The proof of \reff{inter2} is quite similar to that of the previous lemma,
%since the estimate \reff{pq2} still holds if both $p$ and $q$ are odd.
%As for \reff{inter1}, the condition $p\geq 2$ yields that, for any given $\eps>0,$
%\begin{eqnarray*}
% &&m^2\norm{\lambda ^{p+1}  y^{p-1}D_y^{m-1}  v}_{L^2}^2 \\
% &&\leq m^2 \inner{\int_{\abs{y}\leq m/\eps}+\int_{\abs{y}\geq m/\eps}} \lambda^{2(p+1)} \sabs{y}^{2(p-1)} \sabs{D_y^{m-1} v }^2 dxdy \\
%&& \leq\eps^2 \norm{\lambda^{p+1}  y^{p}D_y^{m-1}  v}_{L^2}^2+\eps^{-(4p-2)/3}m^{(4p+4)/3} \lambda^{4(p+1)/3}\norm{
%   \lambda^{(p+1)/3}  y^{(p-2)/3}D_y^{m-1}  v}_{L^2}^2.
%\end{eqnarray*}
%The upper bound for the right-hand side terms can be deduced from \reff{m1} and
%\reff{useful2}. This gives the inequality \reff{inter1}.
%\end{proof}

\begin{lem}\label{lemma5}
For any given $k,\ell$ with $\ell \leq p-1$ and $2\leq k\leq m$, we have
\begin{eqnarray}\label{inter3}
\begin{split}
  \norm{\lambda^{p+1} y^{\ell} D_y^{m-k}
  v}_{L^2}  
\leq &   C_p\inner{ \norm{
   Z_\lambda D_y^{m-k}  v}_{L^2}+\norm{
    D_x^{m-k+1}  v}_{L^2}}\\
   &\quad+ C_p\inner{\norm{  D_y^{m-k+1}  v}_{L^2}+\lambda^{p}\norm{\lambda
    D_y^{m-k}  v}_{L^2}}.
\end{split}
\end{eqnarray}
\end{lem}

%\begin{proof}
% This is just a consequence of \reff{m1},  since we may write
%\begin{eqnarray*}
%  \norm{\lambda^{p+1}   y^{\ell}D_y^{m-k}  v}_{L^2}^2
%   \leq    \norm{
%  \lambda^{p+1}  y^{p} D_y^{m-k}  v}_{L^2}^2+  \lambda^{2(p+1)}\norm{D_y^{m-k}  v}_{L^2}^2
%\end{eqnarray*}
%due to the fact that $\ell\leq p-1$.
%\end{proof}

\begin{proof}[Proof of Proposition \ref{pr1}]
  The first conclusion  for $p,q\geq 2$ in Proposition \ref{pr1}
  follows  from \reff{dy+}, \reff{dx+}, \reff{dy2+}, \reff{dx2+}, \reff{inter1}, \reff{inter2} and\reff{inter3} when we take $\lambda$ large enough. 
  The second conclusion is just a consequence of \reff{dy+}, \reff{dx+}, \reff{dy2+}, \reff{dx2+} and \reff{inter} for $\lambda$ sufficiently large. 
  This ends the proof. 
\end{proof}

\subsection{Global Gevrey hypoellipticity}
By Proposition \ref{pr1} we can follow the arguments in \cite[Proposition 4.1]{WXL} to conclude that $f\in H^{\infty}$ whenever 
$\mathbb{L}_{p,q}f\in H^{\infty}.$ 
For the $H^{\infty}$-solution $\left[\begin{array}{c}f_1\\f_2\end{array}\right]$ of the system \reff{lmeqnset}, we get the following two properties, 
which yield the global $\boldsymbol{\wedge}^\ell G^{\sigma,\tau}$-hypoellipticity. 

\begin{prop}\label{prpkey}
Let  $p,q\geq 2$  and let  $\mathbb{L}_{p,q;\lambda} $  be the operator  given in
\reff{source1}. Denote $$\sigma=\left\{
\begin{array}{lll}
\max\set{(p+1)/2,(q+1)/2},&\quad {\it if~both~}p~{\it and~} q~{\it
  are~odd};\\[3pt]
\max\set{(2p+2)/3,(2q+2)/3},&\quad {\it otherwise}.
\end{array}
\right.$$
Suppose $f_1,f_2 \in H^\infty \bigcap \mathcal H_{Z_\lambda}^\infty$ be such that
\begin{eqnarray*}
  \left\{
\begin{array}{lll}
L_{p,q}f_1-M_{p,q}f_2=g_1\\
L_{p,q}f_2+M_{p,q}f_1=g_2,
\end{array}
\right.
\end{eqnarray*}
where $g_1,g_2 \in C^\infty$ satisfy the condition
\begin{equation*}% \label{fG}
     \forall k \in\mathbb Z_+,\quad \norm{D_x^k
       g_i}_{L^2}+\norm{D_y^k
       g_i}_{L^2}\leq
     M_1^{k+1}\inner{k!}^{\sigma},\quad i=1,2,
\end{equation*}
for some constant $M_1$. Then there exists a constant $C_0$,  depending only on $p,
q,  M_1$ and the constant $C$ given in \reff{useful1},  such that for
given $m\geq 1$,  if $\forall ~k \leq m-1,$
\begin{eqnarray*}
  &&\sum_{i=1}^2\norm{\lambda  D_x^k f_i}_{L^2} +\sum_{i=1}^2\norm{\lambda  D_y^k f_i}_{L^2}
+\sum_{i=1}^2\norm{Z_\lambda  D_x^k f_i}_{L^2} +\sum_{i=1}^2\norm{Z_\lambda  D_y^k f_i}_{L^2}\\
 &&\leq \lambda^{(p+q) k} M^{k+1}\inner{k!}^\sigma
\end{eqnarray*}
 for some constant $M\geq M_1$,  then
\begin{eqnarray*}
&& \sum_{i=1}^2\norm{\lambda  D_x^m f_i}_{L^2} +\sum_{i=1}^2\norm{\lambda  D_y^m f_i}_{L^2}
+\sum_{i=1}^2\norm{Z_\lambda  D_x^m f_i}_{L^2}   +\sum_{i=1}^2\norm{Z_\lambda  D_y^m f_i}_{L^2}\\
 &&\leq ~ C_0 \lambda^{(p+q) m}M^{m}\inner{m!}^\sigma.
 \end{eqnarray*}
As a consequence, if we choose $M\geq M_1+C_0+\sum_{i=1}^2\norm{f_i}_{L^2}$, then  we have by induction that, $  \forall ~k \geq 0,$
\begin{eqnarray*}
 && \sum_{i=1}^2\norm{\lambda  D_x^k f_i}_{L^2} +\sum_{i=1}^2\norm{\lambda  D_y^k f_i}_{L^2}
+\sum_{i=1}^2\norm{Z_\lambda  D_x^k f_i}_{L^2}  + \sum_{i=1}^2\norm{Z_\lambda  D_y^k f_i}_{L^2}\\
  &&\leq \lambda^{(p+q) k} M^{k+1}\inner{k!}^\sigma,
\end{eqnarray*}
and thus $\mathbb{L}_{p,q;\lambda}\colon\boldsymbol{\wedge}^1\mathscr{S}'\longrightarrow\boldsymbol{\wedge}^1\mathscr{S}'$ 
is globally $\boldsymbol{\wedge}^1 G^{\sigma,\sigma}$-hypoelliptic in $\R^2.$
\end{prop}
\begin{proof}
Proposition \ref{pr1} yields immediately the result by induction.
\end{proof}

Similarly we have the following proposition.

\begin{prop}\label{prpkey+}
Let  $q=1$ and $p\geq 1$,  and let  $\mathbb{L}_{p,q;\lambda} $  be the operator  given in
\reff{source1}.
 Suppose $f_1,f_2 \in H^\infty \bigcap \mathcal H_{Z_\lambda}^\infty$ such that
\begin{eqnarray*}
  \left\{
\begin{array}{lll}
L_{p,q}f_1-M_{p,q}f_2=g_1\\
L_{p,q}f_2+M_{p,q}f_1=g_2,
\end{array}
\right.
\end{eqnarray*}
with $g_1,g_2 \in C^\infty$ satisfying the condition
\begin{equation*}
     \forall k,s \in\mathbb Z_+,\quad \norm{D_x^k
       D_y^s
       g_i}_{L^2}\leq
     M_1^{k+s+1}k!\inner{s !}^{(p+1)/2},\quad i=1,2,
\end{equation*}
for some constant $M_1$. Then there exists a constant $C_0$,  depending only on $p,
q,  M_1$ and the constant $C$ given in \reff{useful1},  such that for
given $m\geq 2p+2q$,  if
\begin{eqnarray*}
  \forall ~k \leq m-1,\quad \sum_{i=1}^2\norm{\lambda  D_x^k f_i}_{L^2} +\sum_{i=1}^2  \norm{Z_\lambda  D_x^k f_i}_{L^2}\leq \lambda^{k}M^{k+1}k!, \\ 
\sum_{i=1}^2\norm{\lambda  D_y^k f_i}_{L^2} +\sum_{i=1}^2  \norm{Z_\lambda  D_y^k f_i}_{L^2}\leq\lambda^{p k} M^{k+1}\inner{k!}^{(p+1)/2},
\end{eqnarray*}
 for some constant $M\geq M_1$, then
\begin{equation*}
 \sum_{i=1}^2\norm{\lambda  D_x^m f_i}_{L^2} +\sum_{i=1}^2  \norm{Z_\lambda  D_x^m f_i}_{L^2}\leq C_0 \lambda^{m}M^{m}m!
 \end{equation*}
and
\begin{equation*}
\sum_{i=1}^2\norm{\lambda  D_y^m f_i}_{L^2} +\sum_{i=1}^2  \norm{Z_\lambda  D_y^m f_i}_{L^2}\leq C_0 \lambda^{pm}M^{m}\inner{m!}^{(p+1)/2}.
 \end{equation*}
As a consequence, $\mathbb{L}_{p,q;\lambda}\colon\boldsymbol{\wedge}^1\mathscr{S}'\longrightarrow\boldsymbol{\wedge}^1\mathscr{S}'$ 
is globally $\boldsymbol{\wedge}^1 G^{1, (p+1)/2}$-hypoelliptic in $\R^2.$
\end{prop}

\section{Proof of the main result Theorem \ref{thG}}
The main result can be deduced from the following proposition.
\begin{prop}\label{prpequ}
Let $0\leq\ell\leq 2n$.
Let  $\mathbb{L}_{p,q;\lambda}\colon\boldsymbol{\wedge}^\ell\mathscr{S}'\longrightarrow\boldsymbol{\wedge}^\ell\mathscr{S}'$ 
and $\mathbb{L}_{p,q}\colon\boldsymbol{\wedge}^\ell\mathscr{S}'\longrightarrow\boldsymbol{\wedge}^\ell\mathscr{S}'$ 
be  the twisted Laplacians defined above.   
Then  $\mathbb{L}_{p,q}$  is globally $\boldsymbol{\wedge}^\ell G^\sigma$-hypoelliptic (resp.  $\boldsymbol{\wedge}^\ell H^{m}$-hypoelliptic)  in $\R^{2n}$  if   $\mathbb{L}_{p,q;\lambda}$ is globally
$\boldsymbol{\wedge}^\ell G^\sigma$-hypoelliptic (resp.  $\boldsymbol{\wedge}^\ell H^{m}$-hypoelliptic) in $\R^{2n}$.
\end{prop}

\begin{proof}
We give a proof in the case $\ell=1,$ $n=1$ just for the sake of notational simplicity.
Let  $f_1,f_2 \in L^2 (\R^2)$ be  a solution of system \reff{lmeqnset} for $g_1,g_2 \in C^\infty(\R^2)$.  Then, with $g_1,g_2 \in G^{\sigma}$, resp. $g_1,g_2 \in H^m$, and
$f_{1\lambda},f_{2\lambda},g_{1\lambda},g_{2\lambda}$ the rescaled versions with parameters satisfying system \reff{eqnset},
$\forall \alpha\in\mathbb{Z}_+^2$ we get
\begin{eqnarray*}
  &&\norm{D^\alpha f_{i\lambda}}_{L^2}=\lambda^{-1+\sabs\alpha} \norm{D^\alpha
  f_i}_{L^2},\quad  \norm{D^\alpha g_{i\lambda}}_{L^2}=\lambda^{1 +\sabs\alpha} \norm{D^\alpha
  g_i}_{L^2},\\
&&\norm{Z_{j,\lambda}D^\alpha f_{i\lambda}}_{L^2}=\lambda^{\sabs\alpha} \norm{Z_{j,1}D^\alpha
  f_i}_{L^2},
  \end{eqnarray*}
  and
  \begin{eqnarray*}
  \norm{\mpq D^\alpha f_{i\lambda}}_{L^2}=\lambda^{1+\sabs \alpha}\norm{M_{p,q} D^{\alpha} f_i                                                                                                                    }_{L^2},
\end{eqnarray*}
with $i=1,2$ and $j=1,2$.
A direct check shows that the global
$\boldsymbol{\wedge}^\ell G^\sigma$-hypoellipticity (resp. $\boldsymbol{\wedge}^\ell H^{m}$-hypoellipticity)
of $\mathbb{L}_{p,q }$ is deduced from that of $\mathbb{L}_{p,q,\lambda }.$ This ends the proof.
\end{proof}

\end{document}